\documentclass[12pt, reqno]{amsart}
\usepackage{pifont}
\usepackage{mathrsfs}
\usepackage{geometry}
\usepackage{titletoc}
\usepackage{stix2}
\usepackage{amscd}

\usepackage{amsmath}
\usepackage{amssymb} 
\usepackage{enumitem} 
\usepackage{mathtools} 
\usepackage[table]{xcolor} 
\usepackage[all]{xy} 
\usepackage{tikz} 
\usepackage{tikz-cd}
\usepackage{indentfirst} 
\usepackage{babel} 
\usepackage{setspace} 

\usepackage[colorlinks,linkcolor=red,anchorcolor=green,citecolor=blue]{hyperref} 
\hypersetup{linktocpage = true} 

\usepackage{rotating} 

\usepackage{ytableau} 
\usepackage{longtable} 
\newcolumntype{M}[1]{>{\centering\arraybackslash}m{#1}} 

\usepackage{fancyhdr}

\newcommand\cA{{\mathcal A}}
\newcommand\cB{{\mathcal B}}

\newcommand\cF{{\mathcal F}}

\newcommand\cQ{{\mathcal Q}}

\newcommand{\id}{{\rm id}}
\newcommand{\rk}{{\rm rk}}
\newcommand{\codim}{{\rm codim}}

\newcommand{\im}{{\rm im}}

\newcommand\bp{{\bar\partial}}

\usepackage{amsthm}
\theoremstyle{plain}

\newtheorem{theorem}{Theorem}[section]
\newtheorem{thm}{Theorem}[section]
\newtheorem{lemma}[thm]{Lemma}
\newtheorem{prop}[thm]{Proposition}
\newtheorem{cor}[thm]{Corollary}
\newtheorem{defn}[thm]{Definition}

\theoremstyle{definition}
\newtheorem{example}[thm]{Example}
\newtheorem{remark}[thm]{Remark}

\newcommand{\btheorem}{\begin{thm}}
    \newcommand{\etheorem}{\end{thm}}
\newcommand{\bproposition}{\begin{prop}}
    \newcommand{\eproposition}{\end{prop}}
\newcommand{\bdefinition}{\begin{defn}}
    \newcommand{\edefinition}{\end{defn}}
\newcommand{\bcorollary}{\begin{cor}}
    \newcommand{\ecorollary}{\end{cor}}
\newcommand{\bproof}{\begin{proof}}
    \newcommand{\eproof}{\end{proof}}
\newcommand{\bremark}{\begin{remark}}
    \newcommand{\eremark}{\end{remark}}
\newcommand{\eexample}{\end{example}}
\newcommand{\bexample}{\begin{example}}

\newcommand{\elemma}{\end{lemma}}
\newcommand{\blemma}{\begin{lemma}}

\newcommand{\sq}{\sqrt{-1}}

\newcommand{\p}{\partial}

\renewcommand{\bar}{\overline}

\renewcommand{\phi}{\varphi}

\newcommand{\beq}{\begin{equation}}
\newcommand{\eeq}{\end{equation}}
\newcommand{\ee}{\end{eqnarray*}}
\newcommand{\be}{\begin{eqnarray*}}

\newcommand{\bd}{\begin{enumerate}}
    \newcommand{\ed}{\end{enumerate}}

\renewcommand{\tilde}{\widetilde}

\newcommand{\qtq}[1]{\quad\mbox{#1}\quad}
\renewcommand{\bp}{\bar{\partial}}

\newcommand{\ts}{\otimes}

\renewcommand{\>}{\rightarrow}

\newcommand{\C}{{\mathbb C}}

\newcommand{\R}{{\mathbb R}}

\newcommand{\LL}{\left\langle}
\newcommand{\RL}{\right\rangle}

\renewcommand{\>}{\rightarrow}

\renewcommand{\p}{{\partial}}
\renewcommand{\bp}{{\bar{\partial}}}

\newcommand{\vone}{ \vskip 1\baselineskip}
\newcommand{\om}{\omega}

\renewcommand{\bar}{\overline}
\renewcommand{\tilde}{\widetilde}

\newcommand{\smo}{\sqrt{-1}}

\newcommand{\tr}{\mathrm{tr}}

\renewcommand{\id}{\mathrm{Id}}
\newcommand{\Herm}{\mathrm{Herm}}
\setlist[itemize]{leftmargin=*}
\setlist[enumerate]{leftmargin=*}

\numberwithin{equation}{section} 

\makeatletter

\usepackage{fancyhdr}
\title{The prescribed Hermitian-Yang-Mills flow II}

\author{Zhiyao Xiong}
\author{Xiaokui Yang}
\author{Shing-Tung Yau}

\address{Zhiyao Xiong, Department of Mathematics, Tsinghua University, Beijing, 100084, China}
\email{xiongzy22@mails.tsinghua.edu.cn}

\address{Xiaokui Yang, Department of Mathematics and Yau Mathematical Sciences Center, Tsinghua University, Beijing, 100084, China}
\email{xkyang@mail.tsinghua.edu.cn}

\address{Shing-Tung Yau, Yau Mathematical Sciences Center and  Qiuzhen College, Tsinghua University, Beijing, 100084, China}
\email{styau@mail.tsinghua.edu.cn}

\begin{document}

    \begin{abstract} We prove an analogue of the classical Donaldson–Uhlenbeck–Yau theorem by using  the prescribed Hermitian–Yang–Mills flow.    Let $E$ be a holomorphic vector bundle over a compact K\"ahler manifold $(M,\om_g)$.
        Suppose that for every proper coherent subsheaf $\cF\subset E$, the following inequality holds:
    $$
        \deg_{\omega_g}(\cF)<\deg_{\omega_g}(E).
    $$
        Then, for any initial Hermitian metric $h_0$ on $E$ and any  positive-definite Hermitian tensor $P\in \Gamma(M,E^*\ts \overline E^*)$, the prescribed Hermitian-Yang-Mills flow
    $$  \ \frac{\partial h}{\partial t} = -\Lambda_{\om_g}\left(\sqrt{-1}\, R^h\right) + P,
    $$
         admits a global smooth solution on $[0,\infty)$. Moreover, as $t\>\infty$, the flow converges smoothly  to a Hermitian metric  $h_\infty$ on $E$ satisfying
    $$
        \Lambda_{\om_g}\left(\sqrt{-1}\, R^{h_\infty}\right) = P.
    $$
As an application,  we establish that on a Fano  manifold $M$,  for any Hermitian metric form $\omega$ and any positive-definite Hermitian  tensor  $P\in\Gamma(M,T^{*1,0}M\ts T^{*0,1}M)$, there exists a unique Hermitian metric tensor  $h$ on $T^{1,0}M$ such that 
$$ \Lambda_\omega\left(\sq R^h\right)=P.$$ This may be viewed as an analogue of the Calabi–Yau theorem for Fano manifolds.

    \end{abstract}
    \maketitle {
    \setcounter{tocdepth}{1}

    {\small{    \begin{spacing}{1.1} \tableofcontents %
                \dottedcontents{section}[1.8cm]{}{3em}{5pt} %
\end{spacing} }} }

\vspace*{-1\baselineskip}

\section{Introduction}

This paper continues our previous work \cite{XYY26+}, where we introduced the prescribed Hermitian--Yang--Mills flow
 and established its convergence under various \textbf{differential-geometric} conditions, with applications to solving general prescribed Hermitian--Yang--Mills tensor equations. Here, building on these results, we prove convergence of the flow under \textbf{algebraic-geometric} conditions and thereby obtain solutions to the prescribed Hermitian--Yang--Mills tensor equations in this new setting.

As discussed in our previous works \cite{WYY26+, FWYY26+, WYY26b+, XYY26+}, the prescribed Hermitian-Yang-Mills tensor equation
\beq \Lambda_{\om_g}\left(\sqrt{-1}\, R^h\right) =P\eeq
 serves as a vector bundle analogue of the Calabi–Yau theorem (\cite{Cal57,Yau78}) for prescribed Ricci curvatures. This equation is  closely related to the Hermitian-Einstein equation and extends the classical Hermitian-Einstein existence theory to the setting of fully prescribed Hermitian-Yang-Mills curvature tensors. We anticipate that this equation will emerge as a crucial bridge between the geometry of vector bundles and the broader landscape of K\"ahler and Hermitian geometry.\\

 Let's recall the classical Donaldson--Uhlenbeck--Yau theorem (\cite{Don85, UY86, Don87}) for stable vector bundles (see also the Hermitian analogue in \cite{LY86}):  \setcounter{theorem}{0}
\renewcommand{\thetheorem}{\Alph{theorem}}\begin{theorem} \label{DUY} For  a stable holomorphic vector bundle $E$ over a compact K\"ahler manifold $(M,\omega_g)$,  there exists a unique Hermitian-Einstein metric $h$ on $E$ up to scaling, satisfying
    \beq \Lambda_{\omega_g} \left(\sq R^h\right)=\lambda_E\cdot h, \label{HE} \eeq
    where $R^h\in\Gamma(M,\Lambda^{1,1}T^*M\ts E^*\ts \bar E^*)$ is the Chern curvature tensor of $(E,h)$, and
\beq
\lambda_E=\frac{2\pi n\int_M c_1(E)\wedge
	\omega_g^{n-1}}{\mathrm{\rk}(E)\int_M\omega_g^n}.\eeq  \end{theorem}

\noindent Slope stability is a fundamental algebraic condition:  for every non-zero proper coherent subsheaf $\mathcal{F} \subset E$, the following inequality holds:
\beq 
\frac{\deg_{\omega_g}(\mathcal{F})}{\operatorname{rk}(\mathcal{F})} < \frac{\deg_{\omega_g}(E)}{\operatorname{rk}(E)}.
\label{stable condition}
\eeq 
Here the degree is  defined by
\beq 
\deg_{\omega_g}(\mathcal{F}) = 2\pi n \int_M c^{\mathrm{BC}}_1(\mathcal{F}) \wedge \omega_g^{n-1},
\eeq 
where $c^{\mathrm{BC}}_1(\cF)$ denotes the first Bott--Chern class in $H^{1,1}_{\mathrm{BC}}(M, \mathbb{C})$, coinciding with the usual first Chern class $c_1(\cF)$ when $M$ is K\"ahler. We refer to \cite{Kob87} for more details. 
 Stable bundles and Hermitian-Einstein metrics have become a central topic in differential geometry and mathematical physics; see, for instance, \cite{NS65, Siu87, Sim88, Sim92, BGP96, Li00, DW04, Jac14, LZ15, JW18, Fu19, CS20, DS21, Li21, Wu23, CW24, MPS25}.

 \vskip 1\baselineskip

The first main result of this paper is an analogue of Theorem~\ref{DUY},
established via the prescribed Hermitian-Yang-Mills flow
\begin{equation}
\label{PHYMF}
\begin{cases}
\displaystyle \frac{\partial h}{\partial t}
= -\Lambda_{\omega_g}\!\left(\sqrt{-1}\, R^h\right) + P, \\[10pt]
h(0) = h_0.
\end{cases}
\end{equation}
We solve the prescribed Hermitian-Yang-Mills tensor equation under natural algebraic-geometric conditions, in contrast to earlier differential-geometric assumptions.
\btheorem \label{main}
Let $(M,\omega_g)$ be a compact Hermitian manifold with
$\partial\bar{\partial}\omega_g^{n-1}=0$,
and let $E\to M$ be a holomorphic vector bundle. Suppose that for every  proper coherent subsheaf
$\mathcal{F}\subset E$, the inequality
\begin{equation}
\label{sheaf condition}
\deg_{\omega_g}(\mathcal{F}) < \deg_{\omega_g}(E)
\end{equation}
holds. Then, for any initial Hermitian metric $h_0$ on $E$
and any positive-definite Hermitian tensor
$P\in\Gamma(M,E^*\otimes\overline{E}^*)$,
the prescribed Hermitian–Yang–Mills flow \eqref{PHYMF}
admits a unique global smooth solution on $[0,\infty)$.
Moreover, $h(t)$ converges smoothly as $t\to\infty$
to a Hermitian metric $h_\infty$ on $E$ satisfying
\begin{equation}
\Lambda_{\omega_g}\left(\sqrt{-1}\,R^{h_\infty}\right)=P.
\end{equation}
\etheorem

\noindent By Serre duality, the corresponding statements of Theorem~\ref{main} remain valid for negative bundles; see \cite{WYY26+}.  We  outline the proof of Theorem~\ref{main}. The central ingredient is a uniform $C^0$-estimate for $h_t$. By applying the parabolic comparison principle established in our previous paper \cite{XYY26+}, we obtain the uniform lower bound
\[
C_1 h_0 \leq h_t.
\]
Conversely, a uniform upper bound $h_t \leq C_2 h_0$ would imply convergence of the flow.  It therefore remains to rule out the case of unbounded metrics. Assuming no such uniform bound exists, we combine several a priori estimates along the flow \eqref{PHYMF} with the method of Uhlenbeck--Yau to extract a convergent subsequence of a normalization $\tilde H_t$ of Hermitian matrices $H_t = h_t \cdot h_0^{-1}$. This limiting process yields a proper coherent subsheaf $\mathcal{F} \subset E$ satisfying
$$
\deg_{\omega_g}(\mathcal{F})\geq \deg_{\omega_g}(E),
$$
contradicting condition \eqref{sheaf condition}.\\

It is well-known that on a compact Hermitian manifold $(M,\omega_g)$  with
$\partial\bar{\partial}\omega_g^{n-1}=0$,  if   there exists a Hermitian  metric $h$ on $E$ such that
$$\Lambda_{\om_g}\left(\sqrt{-1}\, R^{h}\right) >0,$$
then the condition \eqref{sheaf condition} is satisfied. As an
immediate consequence of Theorem~\ref{main} and
\cite[Theorem~1.2]{XYY26+}, we obtain: \bcorollary \label{main2} Let
$(M,\om_g)$ be a compact Hermitian manifold with
$\p\bp\omega_g^{n-1}=0$, and $E$ be a holomorphic vector bundle over
$M$. The following are equivalent: \bd \item For every proper
coherent subsheaf $\cF\subset E$, the following inequality holds:
$$
\deg_{\omega_g}(\cF)<\deg_{\omega_g}(E).
$$
\item For  any positive (resp. quasi-positive) Hermitian tensor $P\in \Gamma(M,E^*\ts \overline E^*)$,  there exists a Hermitian metric  $h$ on $E$ satisfying
$$
\Lambda_{\om_g}\left(\sqrt{-1}\, R^{h}\right) = P.
$$
\ed
\ecorollary

 Moreover, Theorem \ref{main} admits the following algebro-geometric formulation. As a consequence, we obtain a criterion for the solvability of the prescribed Hermitian-Yang-Mills tensor equation with respect to an arbitrary Hermitian metric form $\omega$ on $M$:
\btheorem\label{main4} Let $E$ be a holomorphic vector bundle over a compact complex manifold $M$. The following are equivalent.
\bd \item For any nonzero coherent quotient sheaf $\cQ$ of $E$, the determinant line bundle $\det\cQ=:\det(\cQ^{**})$ is pseudo-effective but not unitary flat.

\item For any Hermitian metric $\omega$ on $M$,  and any (quasi-)positive-definite Hermitian tensor $P\in \Gamma(M,E^*\ts \bar E^*)$, there exists a unique Hermitian metric $h$ on $E$ such that 
$$\Lambda_\omega\left(\sq R^h\right)=P.$$
\ed 

\etheorem

\noindent Condition $(1)$ in Theorem \ref{main4} is commonly encountered in algebraic geometry.  It is satisfied, for example, if
\bd \item $E$ is ample (\cite{Har66});
\item  $E$ is  strictly nef and $M$ is a weak Fano manifold (\cite{LOY19, LOY21, LOY24}).
\ed 
Combining Theorem \ref{main4} with  recent results of W.-H. Ou \cite{Ou23}, we obtain existence of solutions to the prescribed Hermitian–Yang–Mills tensor equation on Fano manifolds and certain rationally connected manifolds. This may be viewed as a Fano-manifold analogue of the classical Calabi-Yau theorem.
\btheorem\label{main5} Let $M$ be 
 a rationally connected projective  manifold with $-K_M$ nef.
 Then for any Hermitian metric form $\omega$ on $M$, and for any (quasi-)positive-definite Hermitian tensor  $P\in \Gamma\left(M,T^{*1,0}M\ts T^{*0,1}M\right)$, there exists a unique Hermitian metric  $h$ on $T^{1,0}M$ such that 
\beq \Lambda_\omega\left(\sq R^h\right)=P.\eeq
In particular, the same conclusion holds if $M$ is a projective manifold with $-K_M$ strictly nef, or if $M$ is a Fano manifold.
\etheorem

\noindent 
Building on Theorem \ref{main}, \cite[Corollary~1.5]{Yang18}, \cite[Theorem~1.2]{XYY26+}, and the recent result of Li-Zhang-Zhang \cite[Theorem~1.7]{LZZ25}, we obtain:

\bcorollary\label{main7}  Let $M$ be a compact K\"ahler manifold. Then the following are equivalent:

\bd \item $M$ is a rationally connected projective manifold. 

\item there exists a Hermitian metric form $\omega$ on $M$ and a  Hermitian metric tensor $h_0$ on $T^{1,0}M$ such that $$ \Lambda_\omega\left(\sq R^{h_0}\right)>0.$$
\item there exists a Hermitian metric form $\omega$ on $M$ such that:  for any (quasi-)positive Hermitian tensor  $P\in \Gamma\left(M,T^{*1,0}M\ts  T^{*0,1}M\right)$, there is a unique Hermitian metric tensor $h$ on $T^{1,0}M$ such that 
\beq \Lambda_\omega\left(\sq R^h\right)=P.\label{PHYM}\eeq 
\ed 
\ecorollary
\noindent 
Let $M$ be a rationally connected projective manifold, 
and let $\tilde{M} \to M$ be a sufficiently high blow-up. Then there exists a K\"ahler metric $\omega$ on $\tilde M$ such that for any  positive Hermitian tensor  $P\in \Gamma\left(\tilde M,T^{*1,0}\tilde M\ts  T^{*0,1}\tilde M\right)$, the equation \eqref{PHYM} is not solvable. This highlights a key difference between Theorem~\ref{main5} and Corollary~\ref{main7}. \\

\noindent 
It is well-known (e.g., \cite{Har66}) that ample vector bundles satisfy the condition \eqref{sheaf condition}. Consequently, we obtain:
\bcorollary\label{main3} Let $E$ be an ample vector bundle over a compact complex manifold $M$. Then   for any (quasi-)positive-definite Hermitian tensor $P\in \Gamma(M,E^*\ts \overline E^*)$ and for any Hermitian metric $\omega$ on $M$,  there exists a unique Hermitian metric  $h$ on $E$ satisfying
$$
\Lambda_{\om}\left(\sqrt{-1}\, R^{h}\right) = P.
$$
In particular, every ample vector bundle $E$ over a compact Riemann surface $M$ admits a  Hermitian metric with  Griffiths-positive curvature of the prescribed form $\omega\ts P$, where $\omega$ is a K\"ahler metric on $M$ and $P$ is a positive-definite Hermitian tensor in $\Gamma(M, E^*\ts \bar E^*)$.
\ecorollary

\noindent 
It is also well-known that strictly nef vector bundles arise as extensions of ample bundles. Thanks to \cite[Theorem~1.5]{LOY21} and \cite[Proposition~2.6]{LOY21}, we obtain:

\bcorollary\label{main6}  Let $E$ be a strictly nef vector bundle over a Fano manifold $M$. Then   for any (quasi-)positive-definite Hermitian tensor $P\in \Gamma(M,E^*\ts \overline E^*)$ and for any Hermitian metric $\omega$ on $M$,  there exists a unique Hermitian metric  $h$ on $E$ satisfying
$$
\Lambda_{\om}\left(\sqrt{-1}\, R^{h}\right) = P.
$$
\ecorollary

\bremark  The Griffiths conjecture asserts that every ample vector bundle over a compact complex manifold admits a Griffiths-positive Hermitian metric. It is the vector bundle analogue of Kodaira's embedding theorem, which states that any ample line bundle admits a smooth metric with positive curvature. The conjecture was settled for compact Riemann surfaces through several independent approaches:
\begin{itemize}
    \item Campana and Flenner \cite{CF90} proved it using base-change techniques for algebraic curves (see also \cite{Ume73}).

    \item Li, Zhang, and Zhang \cite{LZZ25} established the result by employing the refined continuity method of Uhlenbeck-Yau \cite{UY86}. They established that the existence of a Hermitian metric with positive Hermitian-Yang-Mills tensor is equivalent to a condition analogous to \eqref{sheaf condition}, which is a special case of Corollary \ref{main2}.

    \item Recently, Murakami \cite{Mur26} provided an alternative proof utilizing the Demailly system proposed in \cite{Dem21}. See also the approaches in \cite{Pin21}.

    \item Our proof provides a complete picture of the Griffiths-positive curvature tensor by constructing a  Hermitian metric with prescribed curvature.
\end{itemize}
 The Griffiths conjecture has also been investigated in greater generality; see, e.g., Berndtsson \cite{Ber09}, as well as Liu, Sun, and the second-named author \cite{LSY13}.  More recently, an RC‑positivity analogue of the Griffiths conjecture was proposed in \cite{Yang18} and  \cite{Yang20}, which also serves as the primary motivation for this series of works.
\eremark

\bremark 
The strategy employed in the proof of Theorem \ref{main} carries over, with appropriate changes, to the twisted Hermitian-Yang-Mills flow:
\beq \frac{\partial h}{\partial t} = -\Lambda_{\omega_g}\left(\sqrt{-1}\, R^h\right) +\lambda h+P. \eeq
This flow approach generalizes  the results in \cite{WYY26b+}.
\eremark
  \noindent\textbf{Acknowledgements}. The second-named author would like to thank Huai-Dong Cao,  Bing-Long Chen, Jixiang Fu, Jie Liu, Kefeng Liu, Wenhao Ou and Valentino Tosatti  for inspiring  discussions.

\vskip 2\baselineskip

\section{Background materials}

In this section, we recall some background material and key results from our previous work \cite{XYY26+} for the reader's convenience.
Let $(M,\omega_g)$ be a compact Hermitian manifold of dimension $n$.
We say that $(M,\omega_g)$ is a \emph{Gauduchon manifold} if
$
\partial\bar{\partial}\,\omega_g^{\,n-1}=0.
$
It is well known that, for any Hermitian metric $\omega$ on $M$,
there exists a smooth function $f$ such that
$
\omega_g = e^{f}\omega
$
satisfies $\partial\bar{\partial}\,\omega_g^{\,n-1}=0$ (\cite{Gau84}). Let $E$ be a holomorphic vector bundle over $M$ of rank $r$. We use the following notations:

\begin{itemize}
    \item $\mathrm{Herm}(E)$: the space of Hermitian tensors in
    $\Gamma(M, E^*\otimes \bar{E}^*)$;

    \item $\Herm^+(E)$: the subspace consisting of
    positive-definite Hermitian tensors;

    \item $\Herm^{\ge 0}(E)$: the subspace consisting of
    non-negative Hermitian tensors.
\end{itemize}
Given a smooth Hermitian metric $h$ on $E$, the Chern connection of
$(E,h)$ is denoted by $\nabla^{h}$. We shall use the natural
decomposition $ \nabla^{h}=\p^{h}+\bp$, where $\p^{h}$ is the
$(1,0)$ part and $\bp$ is the $(0,1)$ part. Let $R^{h}$ be the Chern
curvature tensor of $(E,h)$. In local holomorphic coordinates
$\{z^i\}$ of $M$ and local holomorphic basis $\{e_\alpha\}$ of $E$,
\beq R^h=R_{i\bar j\alpha\bar\beta}dz^i\wedge d\bar z^j\ts
e^\alpha\ts\bar{e}^\beta \in\Gamma\left(M,\wedge^{1,1}T^*M\ts
E^*\ts\bar{E}^*\right). \eeq  The \textbf{Hermitian-Yang-Mills tensor} $ S^{h}\in
\Gamma(M,E^*\ts \bar E^*)$ of $(E,h)$  is defined as \beq S^{h}
:=\Lambda_{\omega_g}\left(\smo R^{h}\right)=\left(g^{i\bar
    j}R^{h}_{i\bar j\alpha\bar\beta}\right) e^\alpha\ts \bar e^\beta \in
\Gamma(M,E^*\ts \bar E^*).
\eeq  By using the Hermitian $h$, there
are natural lifts of $R^h$ and $S^h$: \beq\Theta^h=R^h\cdot
h^{-1}=R_{i\bar j\alpha}^{\beta}dz^i\wedge d\bar z^j\ts
e^\alpha\ts{e}_\beta\in \Gamma(M,\Lambda^{1,1}T^*M\ts E^*\ts E),\eeq
and \beq   K^h=S^h\cdot h^{-1}=\left(g^{i\bar j}R_{i\bar
    j\alpha}^{\beta}\right) e^\alpha\ts  e_\beta\in \Gamma(M,E^*\ts E).
\eeq
The Chern scalar curvature $s_h$ is defined as
\beq s_h=\mathrm{tr}_E K^h=\mathrm{tr}_h S^h=g^{i\bar j}h^{\alpha\bar\beta}R_{i\bar j \alpha\bar\beta}^h.\eeq

\noindent The following transformation formula for the
Hermitian-Yang-Mills curvature tensors is essentially well-known
(e.g. \cite[Proposition~3.6]{WYY26+})
\blemma\label{lem Theta formula} Let $(M,\omega_g)$ be
a compact  Hermitian manifold and $E\>M$ be a holomorphic vector
bundle. If $h$ and $h_0$ are smooth Hermitian metrics on $E$, then
\beq \Lambda_{\omega_g}\left(\sq  \Theta^{h}\right)-
\Lambda_{\omega_g}\left(\sq  \Theta^{h_0}\right)=
\Lambda_{\omega_g}\smo\bp\left( (\p^{h_0} H) \cdot H^{-1}
\right)\label{conformalchange1} \eeq as tensors in $\Gamma(M,E^*\ts
E)$ and $H=h\cdot h_0^{-1}$.
\elemma
 The following results were established in \cite[Section~2]{XYY26+}, and we quote them without proof.
\blemma\label{lem tilde S evolution}The evolution equation of $|\tilde S|^2_h$ along the flow \eqref{PHYMF} is
\beq
\left(\frac{\p }{\p t}-\Delta_\C \right)|\tilde S|_{h}^2
=-2\tilde S_{\alpha\bar\beta}  P_{\gamma\bar\delta}\tilde S_{\lambda\bar \mu} h^{\alpha\bar \mu}  h^{\lambda \bar \delta} h^{\gamma\bar\beta}
- \left|\nabla \tilde S\right|_{g,h}^2,
\label{tilde S evolution}
\eeq
where $\tilde S =S^h-P$, $\Delta_{\C}\bullet =\Lambda_{\omega_g}\left(\sq \p\bp \bullet\right) $ and  $\nabla$ denotes the Chern connection on $E^*\ts \bar E^*$ induced by $(E,h)$.
\elemma

\bcorollary\label{cor tilde S inequality} The following estimate holds:
\beq
\left(\frac{\p }{\p t}-\Delta_\C \right)|\tilde S|_{h}^2 \leq -2 \lambda^P_{\min}\cdot |\tilde S|_{h}^2,
\eeq
where
\beq \lambda^P_{\min}(x,t):=\inf_{v \in E_x, v\neq 0}  \frac{ P(v, \bar v) }{h(t)(v,\bar v)}.\eeq

\ecorollary

\noindent A crucial ingredient in the proof of Theorem \ref{thm finite time extension} is the higher-order estimate for the Hermitian metrics $h_t$ established below.

\btheorem\label{parabolicbootstrap}
Let $(M,\omega_g)$ be a compact Hermitian manifold, and $E$ be a holomorphic vector bundle over $M$.
Suppose that  $h_t$, $0\leq t<T\leq \infty$, is a smooth solution to the flow \eqref{PHYMF}. If there exist constants $C_1\geq 1$ and $C_2>0$ such that
\beq
C_1^{-1}h_0\leq h_t\leq C_1h_0,
\qquad
\sup_M\left|S^{h_t}-P\right|_{h_t}\leq C_2
\eeq
for all $t\in[0,T)$, then for every $\tau\in (0, T)$ and every integer $m\geq 0$, there exists a positive constant $C_{m,\tau}=C_{m,\tau}(\omega_g,h_0,P,C_1,C_2)$ such that
\beq
\left\|h_t\right\|_{C^{m}(M,\omega_g, h_0)}\leq C_{m,\tau},
\qquad t\in[\tau,T).
\label{uniform parabolic bootstrapping estimate}
\eeq
\etheorem

\btheorem\label{thm finite time extension}
Let $(M,\omega_g)$ be a compact Hermitian manifold, and $E$ be a holomorphic vector bundle over $M$.
Suppose that  $h_t$, $0\leq t<T<\infty$, is a smooth solution to the prescribed Hermitian-Yang-Mills flow \eqref{PHYMF}. If there exist constants $C_1$ and $C_2$ such that
\beq
C_1^{-1}h_0\leq h_t\leq C_1h_0,
\qquad
\sup_M\left|S^{h_t}-P\right|_{h_t}\leq C_2
\label{finite extension assumptions}
\eeq
for all $t\in[0,T)$, then $\{h_t\}$ converges in $C^\infty$ to a smooth Hermitian metric $h_T$ as $t\to T$. In particular, the flow \eqref{PHYMF} extends smoothly beyond $T$.
\etheorem

\noindent
The following result is \cite[Theorem~1.1]{XYY26+}:
\btheorem\label{comparison} Let $(M,\om_g)$ be a compact Hermitian manifold, and  $E$ be a holomorphic vector bundle on $M$. Suppose that $\underline h(t)$,   $h(t)$ and $\overline h(t)$ are smooth time-dependent Hermitian metrics on $E$ defined for $t\in [0,T)$ with $T\leq \infty$.  \bd \item[(1)]  If  $\overline h(0)\geq h(0)$ and
\beq \frac{\p \overline h}{\p t}-\frac{\p h}{\p t}
\geq S^{h}-S^{\overline h},\eeq
then for any $t\in [0, T)$, one has \beq \overline h(t)\geq h(t).\eeq
\item[(1')] If  $ \underline h(0)\leq h(0)$ and
\beq \frac{\p \underline h}{\p t}-\frac{\p h}{\p t}
\leq S^{h}-S^{\underline h},\eeq
then for any $t\in [0, T)$, one has \beq  \underline h(t)\leq h(t).\eeq
\ed

\etheorem
\vskip 2\baselineskip

\section{Convergence under the condition $h_t\leq Ch_0$}

In this section, we establish several general results concerning the prescribed Hermitian–Yang–Mills flow \eqref{PHYMF}.
\bproposition\label{prop non-negative P long time existence}
Let  $(E,h_0)$ be a Hermitian vector bundle over a compact Hermitian manifold
$(M,\omega_g)$.  For any $P\in \mathrm{Herm}^{\geq 0}(E)$, the prescribed Hermitian-Yang-Mills flow \eqref{PHYMF} exists for all $t\in[0,\infty)$.
\eproposition

\bproof
Let $[0,T_{\max})$ be the maximal existence interval of the flow \eqref{PHYMF}. Let
\beq
u=|S^{h_t}-P|_{h_t}^2.
\eeq
Since $P\geq 0$, Corollary~\ref{cor tilde S inequality} gives
\beq
\left(\frac{\p}{\p t}-\Delta_\C\right)u\leq0.
\eeq
By the parabolic maximum principle, for any $t\in [0,T_{\max})$,
\beq
\sup_M|S^{h_t}-P|_{h_t}\leq C_4,
\label{uniform theta bound proposition new observation}
\eeq
Here $C_4=C_4(\om_g,h_0,P)=\sup_M|S^{h_0}-P|_{h_0}$.  Since $\p_t h_t=-(S^{h_t}-P)$, for every $x\in M$ and $ v\in E_x\setminus\{0\}$,
\beq
\left|\frac{d}{dt}\log h_t(v,\bar v)\right|
\leq C_4.
\eeq
If $T_{\max}<\infty$, then
\beq
e^{-C_4T_{\max}}h_0\leq h_t\leq e^{C_4T_{\max}}h_0
\quad\text{on }[0,T_{\max}).
\eeq
Together with \eqref{uniform theta bound proposition new observation}, Theorem~\ref{thm finite time extension} extends the solution beyond $T_{\max}$, a contradiction. Hence $T_{\max}=\infty$.
\eproof

\bproposition\label{prop uniform Ric2 bound}
Let  $(E,h_0)$ be a Hermitian vector bundle over a compact Hermitian manifold
$(M,\omega_g)$ and $P\in \mathrm{Herm}^{+}(E)$. Then there exists a constant $C=C(\om_g,h_0,P)>0$ such that for all $t\geq 0$, the following estimates hold along the flow \eqref{PHYMF}:
\beq
C h_0\leq h_t,
 \qtq{and}
\sup_M\left| \Lambda_{\om_g}\left(\sq\Theta^{h_t}\right)\right|_{h_t} \leq C.
\eeq
\eproposition

\bproof
Choose $C_1=C_1(\om_g,h_0,P)\in(0,1)$ so small that
$C_1S^{h_0}<P$.  The stationary metric $\ell_t\equiv C_1h_0$ satisfies
\beq
\frac{\p \ell_t}{\p t}+S^{\ell_t}-P=C_1S^{h_0}-P<0,
\qquad \ell_0<h_0.
\eeq
The parabolic comparison principle in Theorem \ref{comparison} asserts that
\beq
C_1h_0\leq h_t
\label{uniform lower bound from barrier}
\eeq
holds for  all $t\geq0$. This is equivalent to the estimate $$h_t^{-1}\leq C_1^{-1}h_0^{-1}.$$  Together with \eqref{uniform theta bound proposition new observation}, this gives
\beq
\left|\Lambda_{\om_g}\left(\sq\Theta^{h_t}\right)\right|_{h_t}
=|S^{h_t}|_{h_t}
\leq |S^{h_t}-P|_{h_t}+|P|_{h_t}\leq C_4+C_1^{-1}\sup_M|P|_{h_0}.\eeq
This completes the proof.
\eproof

\btheorem \label{thm first case} Let  $E$ be a holomorphic vector bundle over a compact Hermitian manifold
$(M,\omega_g)$.  Let $h_0$ be a Hermitian metric on $E$ and  $P\in \mathrm{Herm}^+(E)$.
If there exists a uniform constant $C_1>0$ such that  the following estimate holds along the flow \eqref{PHYMF}
\beq h_t\leq C_1 h_0,\label{upper bound assumption}\eeq
then the flow \eqref{PHYMF} converges smoothly to a Hermitian metric $h_\infty$ on $E$ satisfying
\beq\Lambda_{\omega_g}\left(\sqrt{-1}R^{h_\infty}\right)=P.\eeq
\etheorem

\bproof
The long-time existence follows from Proposition~\ref{prop non-negative P long time existence}. Let
\beq
C_2:=\min_{x\in M}\lambda_{\min}(Ph_0^{-1})(x)>0.
\eeq
The upper bound \eqref{upper bound assumption} implies
\beq
\lambda_{\min}(Ph_t^{-1})
=\min_{v\ne0}\frac{P(v,\bar v)}{h_t(v,\bar v)}
\geq C_1^{-1}C_2.
\eeq
For $u=|S^{h_t}-P|_{h_t}^2$, Corollary~\ref{cor tilde S inequality} gives
\beq
\left(\frac{\p }{\p t}-\Delta_\C \right)u + 2 C_1^{-1}C_2 u  \leq 0.
\eeq
By comparison with the corresponding ODE, for all $t\geq 0$,
\beq
\sup_Mu(\bullet,t)\leq e^{-2C_1^{-1}C_2t}\sup_Mu(\bullet,0).
\eeq
In particular,
\beq
\left|S^{h_t}-P\right|_{h_t}\leq C_3e^{-C_4t},
\label{exponential decay}
\eeq
where $C_3=C_3(\om_g,h_0,P)=\sup_M\left|S^{h_0}-P\right|_{h_0}$ and $C_4=C_4(\om_g,h_0,P,C_1)$. Moreover,  for every $x\in M$ and $ v\in E_x\setminus\{0\}$,
\beq
\left|\frac{d}{dt}\log h_t(v,\bar v)\right|
\leq C_3e^{-C_4t}.
\eeq
Integrating over $[0,\infty)$, one has the following estimates for all $t\geq 0$
\beq
C_5^{-1}h_0\leq h_t\leq C_5h_0,
\label{two sided bound first case}
\eeq
where  $C_5=C_5(\om_g,h_0,P,C_1)>1$. Since $\partial_t h_t = -(S^{h_t} - P)$, estimates \eqref{exponential decay} and \eqref{two sided bound first case} imply
\begin{equation}
\left\| \frac{\partial h_t}{\partial t} \right\|_{C^0(M,\omega_g, h_0)} \leq C_6  e^{-C_7 t}.
\end{equation}
Thus, for $t \geq s \geq 0$,
\begin{equation}
\| h_t - h_s \|_{C^0; h_0} \leq \int_s^t \left\| \frac{\partial h_\sigma}{\partial \sigma} \right\|_{C^0(M,\omega_g, h_0)} d\sigma \leq 2 C_6  e^{-C_7 s}.
\end{equation}
Hence $h_t$ converges uniformly to a continuous Hermitian tensor $h_\infty$. Taking the limit in \eqref{two sided bound first case}, we conclude that $h_\infty$ is a continuous Hermitian metric on $E$.
Finally, utilizing \eqref{exponential decay} and \eqref{two sided bound first case}, Theorem~\ref{parabolicbootstrap} implies that for every integer $m \geq 0$, there exists a constant $\tilde{C}_m = \tilde{C}_m(\omega_g, h_0, P, C_6, C_{7})$ such that for all $t \geq 1$,
\begin{equation}
\| h_t \|_{C^m(M,\omega_g, h_0)} \leq \tilde{C}_m.
\end{equation}
The Arzel\`a-Ascoli theorem implies that every sequence $t_i \to \infty$ admits a subsequence converging in $C^\infty$ to a smooth Hermitian metric. Since we have already established $C^0$ convergence, the limit must be $h_\infty$. Thus, the entire flow converges to $h_\infty$ in $C^\infty$. Passing to the limit in \eqref{exponential decay} yields
\begin{equation}
S^{h_\infty} = P,
\end{equation}
which is equivalent to $\Lambda_{\omega_g}(\sqrt{-1}\, R^{h_\infty}) = P$. This completes the proof.
\eproof

\vskip 1\baselineskip

\section{The blow-up analysis along the flow}

Let $(M,\om_g)$ be a compact Hermitian manifold, and $(E,h_0)$ be a Hermitian vector bundle on $M$. Suppose that $P\in \mathrm{Herm}^+(E)$.
Let $h_t$ be the solution to the prescribed Hermitian-Yang-Mills flow \eqref{PHYMF} defined for $0\leq t <\infty$.
We define
\beq
H_t=h_t\cdot h_0^{-1}\in\Gamma(M,E^*\ts E).
\eeq
Let  \beq \Lambda_t := \sup_M \lambda_{\max}(H_t)\eeq denote the supremum over $M$ of the largest eigenvalue of $H_t$. It is well known that the eigenvalues of an endomorphism are invariant under a change of basis, and hence are independent of the choice of Hermitian metric.   Define
\beq
\tilde H_t:=\Lambda_t^{-1}\cdot H_t.
\eeq
Then it is clear that $ \tilde H_t\leq \mathrm{Id}_E$ with respect  to both $h_0$ and $h_t$. The following lemma  will be useful for later computations.

\blemma \label{lem pointwise estimate}
Let $(E,h_0)$ be a Hermitian holomorphic  bundle over a complex manifold $M$.
If $H_t\in\mathrm{Herm}^+(E,h_0)\subset \Gamma(M,E^*\ts E)$ is a smooth family of positive $h_0$-Hermitian endomorphisms  on $E$, then for any $\sigma\geq 0$,
\beq
\tr_E\left( \left(\partial_tH_t\right)\cdot H_t^{-1} \cdot\partial_tH_t^\sigma\right) \geq 0.
\eeq
\elemma

\begin{proof}
Fix a point $x \in M$ and choose an $h_0$-orthonormal basis $\{e_\alpha\}$ of $E_x$ such that
\[
H_t(x) = \sum_{\alpha=1}^r \lambda_\alpha \, e^\alpha \otimes e_\alpha .
\]
In this basis, we write
\[
(\partial_t H_t)(x) = \sum_{\alpha,\beta=1}^r K_{\alpha\beta} \, e^\alpha \otimes e_\beta .
\]
    Since $\partial_tH_t$ is also $h_0$-Hermitian, $$K_{\beta\alpha}=\overline{K_{\alpha\beta}}.$$
    A straightforward computation shows that
    \beq
    \left(\partial_tH_t^\sigma\right)(x)=\sum_{\alpha,\beta}\phi^\sigma_{\alpha\beta}K_{\alpha\beta}e^\alpha\ts e_\beta,
    \eeq
    where \beq
    \phi^{\sigma}_{\alpha\beta}:=
    \begin{cases}
        \sigma \lambda_\alpha^{\sigma-1}, & \alpha=\beta,\\[3pt]
        \dfrac{\lambda_\alpha^\sigma-\lambda_\beta^\sigma}{\lambda_\alpha-\lambda_\beta}, &  \alpha\ne \beta, \lambda_\alpha\neq \lambda_\beta ,\\[3pt]
        \sigma \lambda_\alpha^{\sigma-1}, &  \alpha\ne \beta, \lambda_\alpha= \lambda_\beta.
    \end{cases} \label{phisigma}
    \eeq
    It is obvious that $\phi^\sigma_{\alpha\beta}>0$.  Hence
    \beq
    \tr_E\left( \left(\partial_tH_t\right)\cdot H_t^{-1} \cdot\partial_tH_t^\sigma\right)(x)
    =
    \sum_{\alpha,\beta}\lambda_\beta^{-1}\phi^\sigma_{\alpha \beta}|K_{\alpha\beta}|^2
    \ge0.
    \eeq
    The proof is complete.
\end{proof}

\blemma \label{lem local estimate for tilde H}
Let $E$ be a holomorphic vector bundle on a Hermitian manifold $(M,\om_g)$.
Suppose that $h$ and $h_0$ are two Hermitian metrics on $E$ and $H=h\cdot h_0^{-1}$. Then for any $\sigma\in (0,1]$, one has
\beq
\left| \p^{h_0}H^\sigma \cdot H^{-\sigma/2} \right|_{g,h_0}^2
\leq  \frac{1}{\sigma}\Delta_\C\left(\tr_E H^\sigma\right)
+\tr_E\left(\left(\Lambda_{\om_g} \left(\sq\Theta^h\right)-\Lambda_{\om_g}\left(\sq \Theta^{h_0}\right)\right)\cdot H^\sigma\right), \label{estimateA}
\eeq
and
\beq
\left|\bp H^\sigma \right|_{g,h}^2
\leq  \frac{1}{2\sigma}\Delta_\C\left(\tr_E H^{2\sigma}\right)
+\tr_E\left(\left(\Lambda_{\om_g} \left(\sq\Theta^h\right)-\Lambda_{\om_g}\left(\sq \Theta^{h_0}\right)\right)\cdot H^{2\sigma}\right). \label{estimateB}
\eeq
\elemma

\bproof
For any $\rho>0$, Lemma~\ref{lem Theta formula} gives
\begin{eqnarray}
&&\tr_E\left(\left(\Lambda_{\om_g} \left(\sq\Theta^h\right)-\Lambda_{\om_g}\left(\sq \Theta^{h_0}\right)\right)\cdot H^\rho\right)\nonumber\\
&=&\Lambda_{\om_g}\sq\bp\left\{(\p^{h_0}H)\cdot H^{-1},H^\rho\right\}_{h_0}
+\Lambda_{\om_g}\sq\left\{(\p^{h_0}H)\cdot H^{-1},\p^{h_0}H^\rho\right\}_{h_0}.
\label{basic trace identity local estimate}
\end{eqnarray}
Here $\{\cdot,\cdot\}_{h_0}$ denotes contraction  induced by $h_0$. Since
\beq
\left\{(\p^{h_0}H)\cdot H^{-1},H^\rho\right\}_{h_0}
=\tr_E\left((\p^{h_0}H)\cdot H^{\rho-1}\right)
=\frac{1}{\rho}\p\tr_E(H^\rho),
\eeq
and $\Lambda_{\om_g}\sq\bp\p=-\Delta_\C$ on functions, \eqref{basic trace identity local estimate} is reduced to
\beq
\tr_E\left(\left(\Lambda_{\om_g} \left(\sq\Theta^h\right)-\Lambda_{\om_g}\left(\sq \Theta^{h_0}\right)\right)\cdot H^\rho\right)
=-\frac{1}{\rho}\Delta_\C\tr_E(H^\rho)+Q_\rho,
\label{trace identity with Q rho}
\eeq
where
\beq
Q_\rho:=\Lambda_{\om_g}\sq\left\{(\p^{h_0}H)\cdot H^{-1},\p^{h_0}H^\rho\right\}_{h_0}.
\eeq
In particular, \eqref{estimateA} is equivalent to the estimate
\beq Q_\sigma\geq \left| \p^{h_0}H^\sigma \cdot H^{-\sigma/2}\right|_{g,h_0}^2\eeq
and
\eqref{estimateB} is equivalent to the estimate
\beq Q_{2\sigma}\geq \left|\bp H^\sigma \right|_{g,h}^2.\eeq

\noindent Indeed, we fix $x\in M$ and choose an $h_0$-unitary basis
$\{e_\alpha\}$ of $E_x$ which diagonalizes $H(x)$: \beq
H(x)=\sum_{\alpha=1}^r\lambda_\alpha e^\alpha\ts e_\alpha, \qquad
\lambda_\alpha>0. \eeq We also write \beq
\left(\p^{h_0}H\right)(x)=\sum_{\alpha,\beta}K_{\alpha\beta}e^\alpha\ts
e_\beta \eeq where $K_{\alpha\beta}\in T_x^{*1,0}M$.  By
computations in Lemma~\ref{lem pointwise estimate}, we obtain \beq
\p^{h_0}H^\sigma=\sum_{\alpha,\beta}\phi^\sigma_{\alpha\beta}K_{\alpha\beta}e^\alpha\ts
e_\beta, \qquad
\p^{h_0}H^{2\sigma}=\sum_{\alpha,\beta}\phi^{2\sigma}_{\alpha\beta}K_{\alpha\beta}e^\alpha\ts
e_\beta, \label{divided difference for covariant derivative} \eeq
where $\phi^\sigma_{\alpha\beta}$ and $\phi^{2\sigma}_{\alpha\beta}$
are defined in \eqref{phisigma}.  It is easy to see that
\begin{eqnarray} Q_\sigma
=\sum_{\alpha,\beta}\lambda_\beta^{-1}\phi^\sigma_{\alpha\beta}|K_{\alpha\beta}|_g^2
\geq
\sum_{\alpha,\beta}\lambda_\beta^{-\sigma}\left(\phi^\sigma_{\alpha\beta}\right)^2|K_{\alpha\beta}|_g^2
=\left| \p^{h_0}H^\sigma \cdot H^{-\sigma/2}\right|_{g,h_0}^2, \label{Q
sigma first estimate}
\end{eqnarray}
where the inequality follows from  the fact that the inequality
\beq
0\leq \frac{x^\sigma-y^\sigma}{x-y}\leq y^{\sigma-1}
\label{scalar estimate one}
\eeq
holds for any $0<\sigma\leq1$, $x,y>0$ and $x\neq y$.  \\

On the other hand,  note that $\bp H^\sigma$ is the $h_0$-adjoint of $\p^{h_0}H^\sigma$, i.e.,
\beq \bar\p_i \left(H^\sigma\right)_\gamma^\alpha = h_0^{\alpha\bar \lambda}h_{0,\gamma\bar \mu}\left(\overline{\p^{h_0}_i \left(H^\sigma\right)_\lambda^\mu}\right). \label{Hermitianrelation} \eeq
It is easy to compute that
\beq
\left|\bp H^\sigma\right|_{g,h}^2
=\sum_{\alpha,\beta}\lambda_\alpha\lambda_\beta^{-1}
\left(\phi^\sigma_{\alpha\beta}\right)^2|K_{\alpha\beta}|_g^2.
\eeq
Hence, one has
\begin{eqnarray}
Q_{2\sigma}
=\sum_{\alpha,\beta}\lambda_\beta^{-1}\phi^{2\sigma}_{\alpha\beta}|K_{\alpha\beta}|_g^2
\geq \sum_{\alpha,\beta}\lambda_\alpha\lambda_\beta^{-1}
\left(\phi^\sigma_{\alpha\beta}\right)^2|K_{\alpha\beta}|_g^2
=\left|\bp H^\sigma\right|_{g,h}^2,
\end{eqnarray}
where the inequality follows from  the fact that the inequality
\beq
\frac{x^{2\sigma}-y^{2\sigma}}{x-y}
\geq x\left(\frac{x^\sigma-y^\sigma}{x-y}\right)^2
\label{scalar estimate two}
\eeq
holds for any $0<\sigma\leq1$, $x,y>0$  and $x\neq y$.  This completes the proof.
\eproof

\bproposition \label{traceestimate}
Let  $(E,h_0)$ be a Hermitian vector bundle over a compact Hermitian manifold
$(M,\omega_g)$ and $P\in \mathrm{Herm}^{+}(E)$. Then there exists a constant $C_1=C_1(\om_g,h_0,P)>0$ such that for all  $\sigma\in (0,1]$, the following estimates hold along the flow \eqref{PHYMF}:
\beq
\left| \p^{h_0}H_t^\sigma \cdot H_t^{-\sigma/2}\right|_{g,h_0}^2
\leq  \frac{1}{\sigma}\Delta_\C\left(\tr_E H_t^\sigma\right)
+C_1\tr_E H_t^\sigma,
\eeq
and
\beq
\left|\p^{h_0} \tilde H_t^\sigma \right|_{g,h_0}^2
\leq \left|\p^{h_0}\tilde H_t^\sigma \cdot \tilde H_t^{-\sigma/2}\right|_{g,h_0}^2
\leq  \frac{1}{\sigma}\Delta_\C\left(\tr_E \tilde H_t^\sigma\right)+C_1.
\label{tilde H inequality}
\eeq
Moreover, there exists a constant $C_2=C_2(\om_g,h_0,P)>0$ such that for all $t\in [0,+\infty)$,
\beq
\left\|\tilde H_t\right\|_{L^2(M,\omega_g, h_0)}\geq C_2. \label{non-trivial condition}
\eeq

\eproposition

\bproof
Applying Lemma~\ref{lem local estimate for tilde H} with $h = h_t$ yields
\beq
\left|\p^{h_0}H_t^\sigma \cdot H_t^{-\sigma/2}\right|_{g,h_0}^2
\leq  \frac{1}{\sigma}\Delta_\C\left(\tr_E H_t^\sigma\right)
+\tr_E\left(\left(\Lambda_{\om_g} \left(\sq\Theta^{h_t}\right)-\Lambda_{\om_g}\left(\sq \Theta^{h_0}\right)\right)\cdot H_t^\sigma\right).
\eeq
Since the tensor $H_t^{\sigma}$ is self-adjoint
with respect to both $h_0$ and $h_t$,
\beq |H_t^\sigma|_{h_t}=|H_t^\sigma|_{h_0}=\left(\tr_EH_t^{2\sigma}\right)^{1/2}\leq\tr_EH_t^\sigma, \eeq and by Proposition \ref{prop uniform Ric2 bound}, there is a constant
$C=C(\om_g,h_0,P)>0$ such that
\be
&&\left|\tr_E\left(\left(\Lambda_{\om_g} \left(\sq\Theta^{h_t}\right)-\Lambda_{\om_g}\left(\sq \Theta^{h_0}\right)\right)\cdot H_t^\sigma\right)\right|\\
&\leq & \left|\Lambda_{\om_g} \left(\sq\Theta^{h_t}\right)\right|_{h_t}\cdot \left|H_t^\sigma\right|_{h_t}+\left|\Lambda_{\om_g} \left(\sq\Theta^{h_0}\right)\right|_{h_0}\cdot \left|H_t^\sigma\right|_{h_0}\\
&\leq&
C\tr_EH_t^\sigma.
\ee  Hence, we get
\beq
\left|\p^{h_0}H_t^\sigma \cdot H_t^{-\sigma/2}\right|_{g,h_0}^2
\leq  \frac{1}{\sigma}\Delta_\C\left(\tr_E H_t^\sigma\right)
+C\tr_E H_t^\sigma.
\eeq
Multiplying by $\Lambda_t^{-\sigma}$ yields
\beq
\left| \p^{h_0}\tilde H_t^\sigma \cdot \tilde H_t^{-\sigma/2}\right|_{g,h_0}^2
\leq  \frac{1}{\sigma}\Delta_\C\left(\tr_E \tilde H_t^\sigma\right)
+C\tr_E \tilde H_t^\sigma.\label{estimateHsigma}
\eeq
Since $0<\tilde H_t\leq\id_E$, we have $\tr_E\tilde H_t^\sigma\leq\rk(E)$ and
$\tilde H_t^{-\sigma/2}\geq\id_E$. This proves \eqref{tilde H inequality} with
$C_1=C\rk(E)$.

Taking $\sigma=1$  in \eqref{estimateHsigma}, one has
\beq
-\Delta_\C \left(\tr_E \tilde H_t\right)\leq C_5\tr_E\tilde H_t,
\eeq
where $C_5=C_5(\om_g,h_0,P)$. Applying the Moser iteration to this scalar inequality yields
\beq
\sup_M  \tr_E \tilde H_t
\leq C_6\left\|\tr_E \tilde H_t\right\|_{L^2(M,\omega_g, h_0)},
\eeq
where $C_6=C_6(\om_g,h_0,P)$. By construction, we have
\beq \sup_M\tr_E\tilde H_t\geq1.\eeq Therefore
\beq
1
\leq C_6\left\|\tr_E \tilde H_t\right\|_{L^2(M,\omega_g, h_0)}.
\eeq
This proves \eqref{non-trivial condition}
with $C_2=\left(\rk(E)C_6\right)^{-1}$.
The proof is completed. \eproof

\vone
\bproposition \label{prop tilde H estimates}
Let  $(E,h_0)$ be a Hermitian vector bundle over a compact Gauduchon manifold
$(M,\omega_g)$ and $P\in \mathrm{Herm}^{+}(E)$.
Suppose that  $h_t$ is the  solution to the flow \eqref{PHYMF} defined for $0\leq t <\infty$. Then there exists a constant $C=C(\om_g,h_0,P)$ such that for all $t\in [0,+\infty)$ and $\sigma\in(0,1]$,
\beq
\left\| \tilde H_t^\sigma \right\|_{W^{1,2}(M,\omega_g, h_0)}\leq C, \label{tilde H W2p estimate}
\eeq
and
\beq
\left\| \bp\tilde H_t^\sigma \right\|_{L^{2}(M,\omega_g, h_t)}\leq C.
\eeq
\eproposition

\bproof
Since $0<\tilde H_t\leq\id_E$,
\beq
\left\| \tilde H_t^\sigma \right\|_{L^2(M,\omega_g, h_0)}^2
\leq \rk(E)\mathrm{Vol}(M,\om_g). \label{tilde H L2 estimate}
\eeq
Integrating inequality \eqref{tilde H inequality} over $M$
and using the Gauduchon condition $\partial\bar{\partial}\,\omega_g^{n-1}=0$,
we obtain
\begin{equation}
\label{d tilde H L2 estimate}
\bigl\|\partial^{h_0}\tilde H_t^\sigma\bigr\|_{L^2(M,\omega_g, h_0)}^2
\leq  C_1\,\mathrm{Vol}(M,\omega_g).
\end{equation}
Since $\tilde H_t^\sigma$ is Hermitian with respect to $h_0$, i.e.,  $(\tilde H_t^\sigma)^*=\tilde H_t^\sigma$,
we have
\[
\bar{\partial}\tilde H_t^\sigma
= \bigl(\partial^{h_0}\tilde H_t^\sigma\bigr)^*,
\]
and therefore
\begin{equation}
\label{d tilde H equation}
\bigl\|\bar{\partial}\tilde H_t^\sigma\bigr\|_{L^2(M,\omega_g, h_0)}
= \bigl\|\partial^{h_0}\tilde H_t^\sigma\bigr\|_{L^2(M,\omega_g, h_0)}.
\end{equation}
Combining \eqref{d tilde H L2 estimate} and \eqref{d tilde H equation}
yields a uniform $W^{1,2}$-bound
\[
\|\tilde H_t^\sigma\|_{W^{1,2}(M,\omega_g, h_0)} \le C,
\qquad C=C(\omega_g,h_0,P).
\]
Set $h=h_t$ in inequality \eqref{estimateB} of Lemma~\ref{lem local estimate for tilde H}
and multiply by $\Lambda_t^{-2\sigma}$. Then
\beq
\left|\bp \tilde H^\sigma_t \right|_{g,h_t}^2
\leq  \frac{1}{2\sigma}\Delta_\C\left(\tr_E\tilde H_t^{2\sigma}\right)
+\tr_E\left(\left(\Lambda_{\om_g}\left(\sq \Theta^{h_t}\right)-\Lambda_{\om_g}\left(\sq \Theta^{h_0}\right)\right)\cdot \tilde H_t^{2\sigma}\right).
\eeq
By Proposition~\ref{prop uniform Ric2 bound},
\beq \left|\Lambda_{\om_g}\sq\Theta^{h_t}\right|_{h_t}\leq C_2,\eeq  where
$C_2=C_2(\om_g,h_0,P)>0$.  Since $\tilde H_t^{2\sigma}$ is self‑adjoint with respect to both $h_0$ and $h_t$,
and its eigenvalues lie in $(0,1]$, we have
\beq 
\bigl|\tilde H_t^{2\sigma}\bigr|_{h_0}
=
\bigl|\tilde H_t^{2\sigma}\bigr|_{h_t}
\le \sqrt{\operatorname{rk}(E)}.
\eeq 
Noting that $\Lambda_{\omega_g}\left(\sqrt{-1}\,\Theta^{h_0}\right)$ is fixed,
there exists a constant $
C_3=C_3(\omega_g,h_0,P,\operatorname{rk}(E))>0
$
such that
\beq
\left|\bp \tilde H^\sigma_t \right|_{g,h_t}^2
\leq  \frac{1}{2\sigma}\Delta_\C\left(\tr_E\tilde H_t^{2\sigma}\right)+C_3.
\eeq
Integrating over $M$ and using the Gauduchon condition
$\partial\bar{\partial}\,\omega_g^{n-1}=0$,
the Laplacian term drops out. Hence
$$
\bigl\|\bar{\partial}\tilde H^\sigma_t\bigr\|_{L^2(M,\omega_g,h_t)}^2
\le C_3\,\mathrm{Vol}(M,\omega_g),
$$
and therefore
$$
\bigl\|\bar{\partial}\tilde H^\sigma_t\bigr\|_{L^2(M,\omega_g,h_t)}
\le C_1,
$$
where $C_1:=\bigl(C_3\,\mathrm{Vol}(M,\omega_g)\bigr)^{1/2}$.
\eproof

\bproposition  \label{prop new observation}
Let  $(E,h_0)$ be a Hermitian vector bundle over a compact Hermitian manifold
$(M,\omega_g)$. For any $P\in \mathrm{Herm}(E)$,  if we set \beq \theta_t=\Lambda_{\om_g}\left(\sqrt{-1}\Theta^{h_t}\right)-P\cdot h_t^{-1}, \qtq{and} V_t=P\cdot h_t^{-1}, \eeq  then for all $t\geq 0$
\beq
\partial_t\theta_t=g^{i\bar j}\nabla_i\nabla_{\bar j}\theta_t-\theta_tV_t,
\eeq
and
\beq
\left(\p_t-\Delta_\C\right)\tr_E\left(\theta_t^2\right)= -2\left|\bp \theta_t\right|_{g,h_t}^2 -2\tr_E\left(\theta_t^2 V_t\right)\label{ev-theta}
\eeq
where $\nabla$ is the Chern connection on $\mathrm{End}(E)$ induced by $h_t$.
\eproposition

\bproof  We use the following notations:
\beq
\tilde S_t=S^{h_t}-P,
\qquad
K^{h_t}=\Lambda_{\om_g}\left(\sq\Theta^{h_t}\right)=S^{h_t}h_t^{-1}.
\eeq
Then $\theta_t=\tilde S_t h_t^{-1}=K^{h_t}-V_t$.  The proof of Lemma~\ref{lem tilde S evolution} gives
\beq
\frac{\p}{\p t}\tilde S_{\alpha\bar\beta}
= g^{i\bar j}\nabla_i\nabla_{\bar j}\tilde S_{\alpha\bar\beta}
-\tilde S_{\alpha\bar\gamma}h_t^{\delta\bar\gamma}S^{h_t}_{\delta\bar\beta}.
\label{tilde S time derivative formula}
\eeq
Since $\nabla h_t=0$, one deduces that
\beq
\left(\frac{\p}{\p t}\tilde S_t\right)h_t^{-1}
=g^{i\bar j}\nabla_i\nabla_{\bar j}\theta_t-\theta_tK^{h_t}.
\eeq
On the other hand, the flow equation $\p_t h_t=-\tilde S_t$ gives
\beq
\p_t h_t^{-1}=h_t^{-1}\tilde S_t h_t^{-1}.
\eeq
Therefore
\beq
\p_t\theta_t
=\left(\frac{\p}{\p t}\tilde S_t\right)h_t^{-1}
+\tilde S_t\left(\frac{\p}{\p t}h_t^{-1}\right)=g^{i\bar j}\nabla_i\nabla_{\bar j}\theta_t-
\theta_tK^{h_t}+\theta_t^2.
\eeq
Since $K^{h_t}=\theta_t+V_t$, we obtain
\beq
\p_t\theta_t=g^{i\bar j}\nabla_i\nabla_{\bar j}\theta_t-\theta_tV_t.
\eeq
Note also that
\beq
\tr_E(\theta_t^2)=|\tilde S_t|_{h_t}^2.
\eeq
In the formula \eqref{tilde S evolution}, we have
\beq \tilde S_{\alpha\bar\beta}  P_{\gamma\bar\delta}\tilde S_{\lambda\bar \mu} h^{\alpha\bar \mu}  h^{\lambda \bar \delta} h^{\gamma\bar\beta}=\mathrm{tr_E}(\theta_t V_t \theta_t)=\tr_E(\theta_t^2V_t).\eeq
Hence, we obtain
\beq
\left(\p_t-\Delta_\C\right)|\tilde S_t|_{h_t}^2
=-2\tr_E(\theta_t^2V_t)-|\nabla\tilde S_t|_{g,h_t}^2.
\eeq
On the other hand, since $\theta_t$ is $h_t$-Hermitian,
\beq
|\nabla\tilde S_t|_{g,h_t}^2=|\p^{h_t}\theta_t|_{g,h_t}^2+|\bp\theta_t|_{g,h_t}^2=2|\bp\theta_t|_{g,h_t}^2.
\eeq
Hence, we obtain \eqref{ev-theta}.
\eproof

\btheorem \label{thm oscillation}
Let  $(E,h_0)$ be a Hermitian vector bundle over a compact Gauduchon manifold
$(M,\omega_g)$ and $P\in \mathrm{Herm}^{+}(E)$.
Suppose that  $h_t$ is the solution to the flow \eqref{PHYMF} defined for $0\leq t <\infty$. Suppose that there exists a sequence $t_i\to+\infty$ with
\beq
\lim_{i\to\infty}\sup_M\lambda_{\max}(H_{t_i})=+\infty.\label{blow-up}
\eeq
Then
\beq
\lim_{i\to\infty}\inf_M\lambda_{\max}(H_{t_i})=+\infty.\label{blow-up inf}
\eeq
Moreover, for every $\sigma\in (0,1]$,
\beq
\lim_{i\to\infty}\sup_M\tr_E\left(Ph_{t_i}^{-1}\cdot \tilde H_{t_i}^\sigma\right)=0.\label{zero term 1}
\eeq
\etheorem
\begin{proof}
    By Proposition~\ref{prop uniform Ric2 bound}, there are constants
    $C_1=C_1(\om_g,h_0,P)>0$ and $C_2=C_2(\om_g,h_0,P)>0$ such that
    \beq
    \lambda_{\min}(H_t)\geq C_1,
    \qquad
    |S^{h_t}-P|_{h_t}\leq C_2
    \eeq
    for all $t\geq0$.  Hence
    \beq
    |s_{h_t}|=\left|\mathrm{tr}_{h_t}S^{h_t}\right|
    \leq \sqrt{\rk(E)}|S^{h_t}-P|_{h_t}+\tr_{h_t}P
    \leq C_3,
    \eeq
    where $C_3=C_3(\om_g,h_0,P)>0$.  Since $H_t=h_t\cdot h_0^{-1}$, we obtain the following difference between Chern scalar curvatures
    \beq
    \Delta_\C\log \det H_t=-s_{h_t}+s_{h_0}.
    \eeq
    In particular,  $$|\Delta_\C\log\det H_t|\leq C_4$$ for
    $C_4=C_4(\om_g,h_0,P)>0$.  A standard Green function estimate for the complex Laplacian
    $\Delta_{\mathbb{C}}$ on the compact Gauduchon manifold yields
    \beq
    \operatorname{osc}\left(\log\det H_t\right)
    =\sup_M\log\det H_t-\inf_M\log\det H_t\leq C_5,
    \label{finite osc}
    \eeq
    where $C_5=C_5(\om_g,h_0,P)>0$.  Since
    $\lambda_{\min}(H_t)\geq C_1$,
    \beq
    \sup_M \log\det H_t
    \geq (\rk(E)-1)\log C_1+\log\sup_M\lambda_{\max}(H_t).
    \eeq
    The assumption \eqref{blow-up} therefore implies
    \beq \inf_M\det H_{t_i}\to\infty, \eeq  and consequently
    \beq
    \inf_M\lambda_{\max}(H_{t_i})
    \geq\left(\inf_M\det H_{t_i}\right)^{1/\rk(E)}\to\infty.
    \eeq
    This proves \eqref{blow-up inf}.

    Fix $\sigma\in(0,1]$.  At a point $x\in M$, choose an $h_0$-unitary basis
    diagonalizing $H_{t_i}(x)$:
    \beq
    h_{t_i,\alpha\bar\beta}=\lambda_\alpha\delta_{\alpha\beta},
    \qquad
    \lambda_1\leq\cdots\leq\lambda_r.
    \eeq
    Then
    \beq
    \tr_E\left(Ph_{t_i}^{-1}\tilde H_{t_i}^\sigma\right)(x)
    =\sum_{\alpha=1}^r P_{\alpha\bar\alpha}
    \frac{\lambda_\alpha^{\sigma-1}}{\Lambda_{t_i}^\sigma}
    \leq C_1^{\sigma-1}\tr_{h_0}P\cdot \lambda_r^{-\sigma}.
    \eeq
    Taking the supremum over $M$ and using \eqref{blow-up inf}, we obtain
    \eqref{zero term 1}.
\end{proof}

\btheorem \label{prop key estimate for degree comparison}
Let  $(E,h_0)$ be a Hermitian vector bundle over a compact Gauduchon manifold
$(M,\omega_g)$ and $P\in \mathrm{Herm}^{+}(E)$.
Suppose that  $h_t$ is the solution to the flow \eqref{PHYMF} defined for $0\leq t <\infty$. Suppose that
\beq
\limsup_{t\to\infty}\Lambda_t=+\infty.
\eeq
Then there exists a sequence $\{t_i\}$ with $t_i\to+\infty$, such that
\beq
\lim_{i\to\infty}\Lambda_{t_i}=+\infty
\eeq
and  for all $\sigma\in (0,1]$
\beq
\limsup_{i\to\infty} \int_M\tr_E\left(\Lambda_{\om_g}\left(\sq \Theta^{h_{t_i}}\right)\cdot \tilde H_{t_i}^{\sigma}\right)\om_g^n\leq 0.\label{key estimate for degree comparison}
\eeq
\etheorem

\bproof Since $\limsup_{t\to\infty}\Lambda_t=+\infty$, we construct a sequence $\{t_i\}\subset \R^+$ as follows:
choose $T_i\to\infty$ with $\Lambda_{T_i}\to\infty$, then there exists $t_i\in [0,T_i]$ such that
\beq
\Lambda_{t_i}=\max_{0\leq s\leq T_i}\Lambda_s. \label{definitionofLambda}
\eeq
In particular, one has \beq \lim_{i\to\infty}\sup_M\lambda_{\max}(H_{t_i})=\lim_{i\to\infty}\Lambda_{t_i}=+\infty.\eeq
Fix $\sigma\in(0,1]$.  Recall that
\beq
\theta_t=\Lambda_{\om_g}\left(\sqrt{-1}\Theta^{h_t}\right)-Ph_t^{-1},
\qquad V_t=Ph_t^{-1}.
\eeq
We define
\beq T_\sigma(t):=\int_M\tr_E\left(\theta_t\tilde H_t^\sigma\right)\om_g^n.\eeq
We claim that
\beq
\limsup_{i\to\infty}T_\sigma(t_i)\leq0. \label{A sigma chosen time nonpositive}\eeq
Hence, \eqref{key estimate for degree comparison} follows from this claim and \eqref{zero term 1}. To prove the estimate \eqref{A sigma chosen time nonpositive}, we define
\beq
\cA_\sigma(t)=\int_M\tr_E\left(\theta_tH_t^\sigma\right)\om_g^n,
\qquad
\cB_\sigma(t)=\int_M\tr_E\left(H_t^\sigma\right)\om_g^n .
\eeq
Since $\p_tH_t=-\theta_tH_t$,  one can show that
\beq \mathrm{tr}_E\left(\p_t H^\sigma_t\right)=\sigma\mathrm{tr}_E\left(H^{\sigma-1}_t\p_t H_t\right)=-\sigma \tr_E(\theta_tH_t^\sigma),\eeq
and so
\beq
\cB_\sigma'(t)=-\sigma\int_M\tr_E(\theta_tH_t^\sigma)\om_g^n=-\sigma \cA_\sigma(t) \label{Bderivative}
\eeq

For every $t\geq 0$, integration by parts gives
\begin{eqnarray}
&&\int_M\tr_E\left(\left(g^{i\bar j}\nabla_i\nabla_{\bar j}\theta_t\right)
H_t^\sigma\right)\om_g^n\nonumber\\
&=&
n\sqrt{-1}\int_M
\left\{\p^{h_t}\bp\theta_t,H_t^\sigma\right\}_{h_t}
\wedge\om_g^{n-1}\nonumber\\
&=&
n\sqrt{-1}\int_M
\left\{\bp\theta_t,H_t^\sigma\right\}_{h_t}
\wedge\p\left(\om_g^{n-1}\right)
-\int_M
\LL\bp\theta_t,\bp H_t^\sigma\RL_{g,h_t}\om_g^n .
\label{integrate by parts theta direct proof}
\end{eqnarray}
By Proposition~\ref{prop new observation},
\beq
\partial_t\theta_t
=
g^{i\bar j}\nabla_i\nabla_{\bar j}\theta_t-\theta_tV_t .
\label{theta evolution direct proof}
\eeq
Differentiating $\cA_\sigma(t)$, and using \eqref{theta evolution direct proof} and
\eqref{integrate by parts theta direct proof}, we obtain
\begin{eqnarray}
\cA_\sigma'(t)&=&\int_M\tr_E\left(\p_t\theta_t\cdot H_t^\sigma\right)\om_g^n+\int_M\tr_E\left(\theta_t\cdot \p_tH_t^\sigma\right)\om_g^n\nonumber\\
&=&
n\sqrt{-1}\int_M
\left\{\bp\theta_t,H_t^\sigma\right\}_{h_t}
\wedge\p\left(\om_g^{n-1}\right)
-\int_M
\LL\bp\theta_t,\bp H_t^\sigma\RL_{g,h_t}\om_g^n
\nonumber\\
&&
-\int_M\tr_E\left(\theta_tV_tH_t^\sigma\right)\om_g^n
+\int_M\tr_E\left(\theta_t\partial_tH_t^\sigma\right)\om_g^n .
\label{cA derivative direct proof}
\end{eqnarray}
Since $\partial_tH_t=-\theta_tH_t$, Lemma~\ref{lem pointwise estimate} implies
\beq
\tr_E\left(\theta_t\partial_tH_t^\sigma\right)
=
-\tr_E\left((\partial_tH_t)H_t^{-1}\partial_tH_t^\sigma\right)\leq 0.
\eeq
Recall that $H_t^\sigma=\Lambda_t^\sigma\tilde H_t^\sigma$.
Moreover, since $0<\tilde H_t^\sigma\leq\id_E$,
we conclude that
\be
\left|
n\sqrt{-1}\int_M
\left\{\bp\theta_t,\tilde H_t^\sigma\right\}_{h_t}
\wedge \p(\om_g^{n-1})
\right|
&=&c_n\left|\int_M
\left\langle \left\{\bp\theta_t,\tilde H_t^\sigma\right\}_{h_t}, *\p\omega_g^{n-1}\right\rangle
\om_g^n\right|  \nonumber\\
&\leq&c_n\int_M\left(  \left|*\p\omega_g^{n-1}\right|_g \cdot |\bp\theta_t|_{g,h_t}\cdot |\tilde H_t^\sigma|_{h_t}\right) \om_g^n \nonumber\\
&\leq&
C_1
\|\bp\theta_t\|_{L^2(M,\omega_g, h_t)},
\label{torsion term C1 estimate}
\ee
where $C_1=C_1(\om_g,\rk(E))>0$.
Hence,
 \eqref{cA derivative direct proof}
implies
\beq
\cA_\sigma'(t)\leq \Lambda_t^\sigma D_\sigma(t)
\label{cA derivative inequality direct proof}
\eeq
for every $t\geq 0$, where
\beq
D_\sigma(t)=
C_1\left\|\bp\theta_t\right\|_{L^2(M,\omega_g, h_t)}
+\left|
\int_M
\LL\bp\theta_t,\bp\tilde H_t^\sigma\RL_{g,h_t}\om_g^n
\right|
+
\left|
\int_M
\tr_E\left(\theta_tV_t\tilde H_t^\sigma\right)\om_g^n
\right| .
\eeq
By Proposition~\ref{prop tilde H estimates}, there is
$C_2=C_2(\om_g,h_0,P)>0$ such that
\beq
\left|\int_M\LL\bp\theta_t,\bp\tilde H_t^\sigma\RL_{g,h_t}\om_g^n\right|
\leq  \left\| \bp\tilde H_t^\sigma \right\|_{L^{2}(M,\omega_g, h_t)}\cdot  \|\bp\theta_t\|_{L^2(M,\omega_g, h_t)}\leq   C_2\|\bp\theta_t\|_{L^2(M,\omega_g, h_t)}.
\eeq
Moreover, the pointwise Cauchy-Schwarz inequality gives
\beq
\left|\tr_E(\theta_tV_t\tilde H_t^\sigma)\right|
\leq
\left(\tr_E(\theta_t^2V_t)\right)^{1/2}
\left(\tr_E(V_t\tilde H_t^{2\sigma})\right)^{1/2}.
\eeq
Since $0 \le \widetilde{H}_t^{2\sigma} \le \mathrm{id}_E$ and
Proposition~\ref{prop uniform Ric2 bound} gives $$h_t \ge C_3\,h_0$$
for some constant $C_3 = C_3(\omega_g, h_0, P) > 0$,
there exists a constant $C_4 = C_4(\omega_g, h_0, P) > 0$ such that
\beq \mathrm tr_E\left(V_t \widetilde{H}_t^{2\sigma}\right) \leq
C_4 .
\eeq
Consequently,
\beq
D_\sigma(t)
\leq C_5\|\bp\theta_t\|_{L^2(M,\omega_g, h_t)}
+C_5\left(\int_M\tr_E(\theta_t^2V_t)\om_g^n\right)^{1/2},
\label{D sigma direct bound}
\eeq
where $C_5=C_5(\om_g,h_0,P)>0$.
Integrating \eqref{ev-theta} over $M$ gives
\beq
-\frac12\frac{d}{dt}\int_M\tr_E(\theta_t^2)\om_g^n
=
\int_M|\bp\theta_t|_{g,h_t}^2\om_g^n
+
\int_M\tr_E(\theta_t^2V_t)\om_g^n.
\eeq
After integrating in time, the right-hand side has finite integral on
$[0,\infty)$. Hence \eqref{D sigma direct bound} implies
\beq
D_\sigma\in L^2(0,\infty).
\label{D sigma L2 direct}
\eeq

We now prove \eqref{A sigma chosen time nonpositive}, i.e.,
\beq \limsup_{i\to\infty}T_\sigma(t_i)=
\limsup_{i\to\infty}
\Lambda_{t_i}^{-\sigma}\cA_\sigma(t_i)
\leq 0.
\label{normalized cA chosen time nonpositive}
\eeq
Suppose, for contradiction, that this fails.
After passing to a subsequence, there exists a constant $\delta > 0$ such that
\beq \Lambda_{t_i}^{-\sigma}\cA_\sigma(t_i)\geq 2\delta
\label{positive normalized cA contradiction assumption}
\eeq
holds for all $t_i$.
Choose $\rho>C_6/(\sigma\delta)$, where
\beq
C_6:=\rk(E)\mathrm{Vol}(M,\om_g).
\eeq
Since $D_\sigma\in L^2(0,\infty)$ and $t_i\to\infty$, Cauchy-Schwarz gives
\beq
\int_{t_i-\rho}^{t_i}D_\sigma(s)ds
\leq
\rho^{1/2}
\left\|D_\sigma\right\|_{L^2([t_i-\rho,t_i])}
\to0.
\eeq
Thus, for all large $i$, one has
\beq
t_i>\rho,
\qquad
\int_{t_i-\rho}^{t_i}D_\sigma(s)ds\leq\delta.
\label{small D interval}
\eeq
For $t\in[t_i-\rho,t_i]$, integrating the first inequality in
\eqref{cA derivative inequality direct proof} from $t$ to $t_i$ gives
\begin{eqnarray}
\cA_\sigma(t)
&\geq&
\cA_\sigma(t_i)-\int_t^{t_i}\Lambda_s^\sigma D_\sigma(s)ds \nonumber\\
&\geq&
2\delta\Lambda_{t_i}^\sigma
-\Lambda_{t_i}^\sigma
\int_{t_i-\rho}^{t_i}D_\sigma(s)ds \nonumber\\
&\geq&
\delta\Lambda_{t_i}^\sigma,
\label{cA lower on interval}
\end{eqnarray}
where we use \eqref{definitionofLambda}, \eqref{positive normalized cA contradiction assumption},
and \eqref{small D interval}. Therefore, by \eqref{Bderivative},
\beq
\cB_\sigma(t_i)-\cB_\sigma(t_i-\rho)
=
-\sigma\int_{t_i-\rho}^{t_i}\cA_\sigma(s)ds
\leq
-\sigma\delta\rho\Lambda_{t_i}^\sigma.
\eeq
Since $\cB_\sigma(t_i)\geq0$, it follows that
\beq
\cB_\sigma(t_i-\rho)\geq \sigma\delta\rho\Lambda_{t_i}^\sigma.
\label{cB lower contradiction}
\eeq
On the other hand, since $0<\tilde H_t^\sigma\leq\id_E$, we have
\beq
\cB_\sigma(t)
=
\Lambda_t^\sigma\int_M\tr_E(\tilde H_t^\sigma)\om_g^n
\leq
C_6\Lambda_t^\sigma .
\eeq
Using \eqref{definitionofLambda}, we get
\beq
\cB_\sigma(t_i-\rho)
\leq
C_6\Lambda_{t_i-\rho}^\sigma
\leq
C_6\Lambda_{t_i}^\sigma.
\label{cB upper contradiction}
\eeq
Equations \eqref{cB lower contradiction} and \eqref{cB upper contradiction}
contradict the choice $\rho>C_6/(\sigma\delta)$. Hence
\eqref{normalized cA chosen time nonpositive} holds.\eproof

\vskip 2\baselineskip

\section{Proof of Theorem \ref{main}}

Let $(M,\om_g)$ be a compact Gauduchon manifold and $(E,h_0)$ be a Hermitian vector bundle on $M$. For any positive integer $m$, let
\beq W^{m,2}(M, E^*\ts E)\eeq
be the space of $W^{m,2}$ sections of $E^*\ts E$ with respect to the metrics $\omega_g$ and $h_0$. We follow the classical setting in \cite{UY86}.
Recall that an element $\pi\in W^{1,2}(M,E^*\ts E)$ is called a \emph{weakly holomorphic projection} of $E$ if the following identities
\beq
\pi^*=\pi=\pi^2\eeq and \beq  \left(\mathrm{Id}_E-\pi\right)\circ\bp\pi =0
\eeq
hold almost everywhere on $M$ where the adjoint is taken with respect to $h_0$. Uhlenbeck and Yau proved in \cite{UY86}  that every weakly holomorphic projection represents a coherent subsheaf $\cF$ of $E$:

\blemma[Uhlenbeck-Yau] \label{lem weakly holomorphic subbundle}
Let $(E,h_0)$ be a Hermitian holomorphic vector bundle over a compact Hermitian manifold $(M, \om_g)$.
If  $\pi\in W^{1,2}(M,E^*\ts E)$ is a weakly holomorphic projection,
then there exists a coherent subsheaf $\cF$ of $E$, and an analytic subset $\Sigma\subset M$ with the following properties:
\bd
\item $\codim_M \Sigma\geq 2$;
\item The restriction $\pi|_{M \setminus \Sigma}$ is smooth and the identities  $\pi^* = \pi = \pi^2$ and $$\left(\mathrm{Id}_E-\pi\right)\circ\bp\pi= 0$$ hold   pointwisely on $M \setminus \Sigma$;
\item $\cF|_{M\setminus \Sigma}=\im\left(\pi|_{M\setminus \Sigma}\right)$ is a holomorphic subbundle of $E|_{M\setminus \Sigma}$.
\ed
\elemma

\bproof Since the proof is purely local, the K\"ahler metric in the original argument (\cite{UY86}) can be replaced by a Hermitian metric.
\eproof

\noindent We obtain the following result along the flow \eqref{PHYMF}, which fits the setting of weakly holomorphic projections.
\btheorem \label{thm convergent subsequence}
Let $(E, h_0)$ be a Hermitian vector bundle over a compact Gauduchon manifold
$(M, \omega_g)$, and let $P \in \operatorname{Herm}^+(E)$.
Suppose that $h_t$ is the solution to the flow \eqref{PHYMF} defined for $0 \leq t < \infty$.
Assume that
\begin{equation}
\limsup_{t\to\infty} \Lambda_t = +\infty,
\end{equation}
and let $\{t_i\}$ be the sequence chosen in Theorem~\ref{prop key estimate for degree comparison}.
Then there exist a subsequence of $\{t_i\}$, again denoted by $\{t_i\}$, and a sequence $\{\sigma_j\}$
with $\sigma_j \in (0,1]$ satisfying $\sigma_j \to 0$, such that:
\bd
\item $\tilde H_{t_i}\> \tilde H_\infty$ in the weak  $W^{1,2}(M,E^*\ts E)$ sense for some non-zero $\tilde H_\infty$ in $W^{1,2}(M,E^*\ts E)$.
Moreover, for each fixed $\sigma_j$, $\tilde H_{t_i}^{\sigma_j}\> \tilde H^{\sigma_j}_\infty$ in the weak  $W^{1,2}(M,E^*\ts E)$ sense.
\item $\tilde H_\infty^{\sigma_j}\> \tilde H$ in the weak $W^{1,2}(M,E^*\ts E)$ sense for some $\tilde H\in W^{1,2}(M,E^*\ts E)$.
\item $\pi:=\mathrm{Id}_E-\tilde H$ is a weakly holomorphic projection of $E$.
\ed
\etheorem

\bproof By Proposition \ref{traceestimate} and Proposition \ref{prop tilde H estimates},
there exist  constants $C_1=C_1(\om_g,h_0,P)$ and $C_2=C_2(\om_g,h_0,P)$ such that for all $t\in [0,+\infty)$ and $\sigma\in(0,1]$,
\beq
\left\| \tilde H_t^\sigma \right\|_{W^{1,2}(M,\omega_g, h_0)}\leq C_1, \label{weakcompactness}
\eeq
and
\beq
\left\|\tilde H_t\right\|_{L^2(M,\omega_g, h_0)}\geq C_2. \label{lowerbound}
\eeq
Hence, assertions (1) and (2) follow from \eqref{weakcompactness} and the classical weak compactness theorem. The convergence $\tilde{H}_{t_i} \rightarrow \tilde{H}_\infty$ in $L^2(M, E^* \otimes E)$, together with \eqref{lowerbound}, implies that $\tilde{H}_\infty$ is nonzero. For (3), the identities $\pi^* = \pi = \pi^2$ hold almost everywhere. These follow from the bound $\tilde{H}_\infty \leq \mathrm{Id}_E$ and the convergence $\tilde{H}_\infty^{\sigma_j} \rightarrow \tilde{H}$ in $L^2(M, E^* \otimes E)$. The relation $\left(\mathrm{Id}_E-\pi\right)\circ\bp\pi = 0$ a.e. follows from Proposition~\ref{traceestimate}.
\eproof

\bproposition \label{prop degree estimate}
Retain the notation of Theorem~\ref{thm convergent subsequence}. Assume that
\begin{equation}
\sup_M \lambda_{\max}(H_{t_i}) \to +\infty.
\end{equation}
Then we have
\begin{equation}
 \int_M \left|\partial^{h_0} \pi\right|_{g,h_0}^2 \om_g^n\leq
\int_M \tr_E\left(\Lambda_{\om_g}\sq\Theta^{h_0}\cdot\left(\pi-\mathrm{Id}_E\right)\right) \om_g^n.
\label{key deg estimate}
\end{equation}
\eproposition

\begin{proof}
    As in the proof of Proposition \ref{traceestimate},  one has
    \[
    \left|\p^{h_0}\tilde H_{t_i}^{\sigma_j}\right|_{g,h_0}^2
    \leq  \frac{1}{\sigma_j}\Delta_\C\left(\tr_E \tilde H_{t_i}^{\sigma_j}\right)
    +\tr_E\left(\left(\Lambda_{\om_g} \left(\sq\Theta^{h_{t_i}}\right)-\Lambda_{\om_g}\left(\sq \Theta^{h_0}\right)\right)\cdot \tilde H_{t_i}^{\sigma_j}\right)
    \]
    Integrating over $M$, we obtain
    \begin{eqnarray}
    \int_M  \left|\p^{h_0}\tilde H_{t_i}^{\sigma_j}\right|_{g,h_0}^2  \om_g^n
    &\leq& -\int_M\tr_E\left( \Lambda_{\om_g}\left(\sq \Theta^{h_0}\right) \cdot \tilde H_{t_i}^{\sigma_j}\right)\om_g^n\nonumber\\
    &&+ \int_M\tr_E\left(\Lambda_{\om_g}\left(\sq \Theta^{h_{t_i}}\right)\cdot \tilde H_{t_i}^{\sigma_j}\right)\om_g^n.
    \end{eqnarray}
    By Theorem \ref{prop key estimate for degree comparison}, for any $\sigma_j\in(0,1]$,  one has
    \beq
    \limsup_{i\to\infty} \int_M\tr_E\left(\Lambda_{\om_g}\left(\sq \Theta^{h_{t_i}}\right)\cdot \tilde H_{t_i}^{\sigma_j}\right)\om_g^n\leq 0,
    \eeq
    and so
    \beq
    \limsup_{j\to\infty}\limsup_{i\to\infty}
    \left[\int_M  \left|\p^{h_0}\tilde H_{t_i}^{\sigma_j}\right|_{g,h_0}^2  \om_g^n
    +\int_M\tr_E\left( \Lambda_{\om_g}\left(\sq \Theta^{h_0}\right) \cdot \tilde H_{t_i}^{\sigma_j}\right)\om_g^n
    \right]
    \leq 0.\label{deg estimate key}
    \eeq

\noindent   On the other hand, by Theorem \ref{thm convergent subsequence}, the weak convergence in $W^{1,2}(M,E^*\ts E)$ implies that
    \beq
    \pi = \lim_{j\to\infty}\lim_{i\to\infty}\left(\mathrm{Id}_E-\tilde H_{t_i}^{\sigma_j}\right)
    \eeq
    in $L^2(M,E^*\ts E)$. In particular,
    \beq
    \int_{M} \tr_E\left(\Lambda_{\om_g}\sq\Theta^{h_0}\cdot\left(\pi-\mathrm{Id}_E\right)\right)\om_g^n
    = -\lim_{j\to\infty}\lim_{i\to\infty}\int_{M}
    \tr_E\left(\Lambda_{\om_g}\sq\Theta^{h_0}\cdot\tilde H_{t_i}^{\sigma_j}\right)\om_g^n.\label{deg estimate 1}
    \eeq
    From the weak convergence in Theorem \ref{thm convergent subsequence}, we have
    \begin{eqnarray}
    \left\|\p^{h_0}\pi\right\|_{L^2(M,\omega_g, h_0)}^2
    &=& \left(\p^{h_0}\pi,\p^{h_0}\pi\right)_{L^2(M,\omega_g, h_0)}\\
    &=& \lim_{j\to\infty}\lim_{i\to\infty}
    \left(\p^{h_0}\pi,\p^{h_0}\left(\mathrm{Id}_E-\tilde H_{t_i}^{\sigma_j}\right)\right)_{L^2(M,\omega_g, h_0)}\nonumber\\
    &\leq& \left\|\p^{h_0}\pi\right\|_{L^2(M,\omega_g, h_0)}
    \cdot \limsup_{j\to\infty}\limsup_{i\to\infty} \left\|\p^{h_0} \tilde H_{t_i}^{\sigma_j} \right\|_{L^2(M,\omega_g, h_0)}.
    \end{eqnarray}
    This implies
    \beq
    \int_M \left|\partial^{h_0} \pi\right|_{g,h_0}^2 \om_g^n    \leq
    \limsup_{j\to\infty}\limsup_{i\to\infty}\int_M \left|\p^{h_0} \tilde H_{t_i}^{\sigma_j} \right|_{g,h_0}^2  \om_g^{n}.\label{deg estimate 2}
    \eeq
    From \eqref{deg estimate key}, \eqref{deg estimate 1} and \eqref{deg estimate 2}, we establish
    \beq  \int_M \left|\partial^{h_0} \pi\right|_{g,h_0}^2 \om_g^n\leq
    \int_M \tr_E\left(\Lambda_{\om_g}\sq\Theta^{h_0}\cdot\left(\pi-\mathrm{Id}_E\right)\right) \om_g^n.
        \eeq
    This completes the proof.
\end{proof}

\vone
\btheorem   Let $(M,\om_g)$ be a compact Gauduchon manifold, and $E$ be a holomorphic vector bundle over $M$.
Suppose that for every  coherent subsheaf $\cF\subset E$ with $\rk(\cF)<\rk(E)$, the inequality holds:
\beq
\deg_{\omega_g}(\cF)<\deg_{\omega_g}(E).
\eeq
Then, for any Hermitian metric $h_0$ on $E$ and any positive-definite tensor $P\in \Gamma(M,E^*\ts \overline E^*)$, the prescribed Hermitian-Yang-Mills flow \eqref{PHYMF} admits a global smooth solution on $[0,\infty)$. Moreover, the flow converges smoothly as $t\>\infty$ to a Hermitian metric  $h_\infty$ satisfying
\beq
\Lambda_{\om_g}\left(\sqrt{-1}\, R^{h_\infty}\right) = P.
\eeq
\etheorem

\begin{proof} Case $1$. If
    \beq \limsup_{t\>\infty}\Lambda_t = \limsup_{t\>\infty}\sup_M \lambda_{\max}(H_t)<+\infty,\eeq  then there exists a uniform constant $C_1>0$ such that  the following estimate holds along the flow \eqref{PHYMF}
    \beq h_t\leq C_1 h_0.\eeq
By Theorem \ref{thm first case},  the flow \eqref{PHYMF} converges as desired.

Case $2$.  Suppose that \beq \limsup_{t\>\infty}\Lambda_t = \limsup_{t\>\infty}\sup_M \lambda_{\max}(H_t)=+\infty.\eeq
    We choose sequences $\{t_i\}$ and $\{\sigma_j\}$ as in Theorem \ref{thm convergent subsequence}.
    In particular, by Lemma \ref{lem weakly holomorphic subbundle},  there exists a weakly holomorphic projection $\pi$ which  represents a coherent subsheaf $\cF$ of $E$.
    Let $\Sigma\subset M$ be the analytic subset as in Lemma \ref{lem weakly holomorphic subbundle}. Then $\Sigma$ has measure $0$ and $\pi$ is smooth and has constant rank on $M\setminus \Sigma$.
 Moreover, Theorem \ref{thm convergent subsequence} implies that $\tilde H_{t_i} \to \tilde H_\infty$ and $\tilde H_\infty^{\sigma_j}\to \tilde H$ in $L^2(M,E^*\ts E)$.
    By choosing a further subsequence, we may assume
    \beq
    \tilde H_\infty=\lim_{i\to\infty}\tilde H_{t_i} \qtq{and}
    \tilde H = \lim_{j\to\infty}\tilde H_\infty^{\sigma_j}
    \eeq
    almost everywhere on $M$.
    Since $\tilde H_\infty$ is non-zero in $W^{1,2}(M, E^*\ts E)$, it has a strictly positive eigenvalue on a set of positive measure in $M \setminus \Sigma$. We deduce that $\rk (\tilde H)\geq 1$ on $M\setminus \Sigma$. Moreover,
    from our construction,  eigenvalues of $\tilde H$ lie in $\{0,1\}$.
    This implies
    \beq
    \rk(\cF)=\rk(\pi)=\rk\left(\mathrm{Id}_E-\tilde H|_{M\setminus \Sigma}\right)\leq \rk(E)-1.
    \eeq
    By Lemma \ref{lem weakly holomorphic subbundle}, $\pi|_{M\setminus \Sigma}$ coincides with the $h_0$-orthogonal projection from $E|_{M\setminus \Sigma}$ to $\cF|_{M\setminus \Sigma}$.
   By the second fundamental form computation, the degree of $\cF$ is given by:
    \beq
    \deg_{\omega_g}(\cF)
    = \int_{M\setminus \Sigma} \tr_{E}\left(\Lambda_{\om_g}\left(\sq\Theta^{h_0}\right)\cdot \pi\right)\om_g^n
    -\int_{M\setminus \Sigma} \left|B\right|_{g,h_0}^2 \om_g^n,
    \eeq
where $B=\left(\mathrm{Id}_E-\pi\right)\circ \left(\p^{h_0}\pi\right)\circ \iota$ is the second fundamental form of the inclusion $\iota: \cF\>E$. A simple calculation shows
\beq |B|^2_{g,h_0}=|\p^{h_0}\pi|^2_{g, h_0}.\eeq
Hence, we obtain
    \beq
   \deg_{\omega_g}(\cF)
    = \int_{M\setminus \Sigma} \tr_{E}\left(\Lambda_{\om_g}\left(\sq\Theta^{h_0}\right)\cdot \pi\right)\om_g^n
    -\int_{M\setminus \Sigma} \left|\partial^{h_0} \pi\right|_{g,h_0}^2 \om_g^n.
    \eeq
    By the estimate \eqref{key deg estimate} in Proposition \ref{prop degree estimate},
    \beq
    \deg_{\omega_g}(\cF) \geq \int_M \tr_E\left(\Lambda_{\om_g}\sq\Theta^{h_0}\right)  \om_g^n= \deg_{\omega_g}(E),
    \eeq
    and this is a contradiction. The proof is completed.
\end{proof}

\bremark As pointed out in \cite{WYY26+}, by Serre duality, the results of Theorem~\ref{main} extend to negative bundles.
\eremark

\vskip 2\baselineskip

\section{Applications in algebraic geometry}
In this section, we discuss applications of Theorem \ref{main} in algebraic geometry and prove Theorem \ref{main4}, Theorem \ref{main5}, and Corollary \ref{main7}.\\

\bproof[Proof of Theorem \ref{main4}]
We recall the classical characterization of pseudo‑effectivity on compact complex manifolds (\cite{Buc99, Lam99}, see also \cite[Theorem~3.4]{Yang19}): 
a holomorphic line bundle $L$ is pseudo‑effective if and only if
\beq 
\int_M c_1^{\mathrm{BC}}(L) \wedge \omega_{\mathrm{G}}^{\,n-1} \;\ge\; 0
\eeq
for every Gauduchon metric $\omega_{\mathrm{G}}$ on $M$. 
Moreover,  $L$ is pseudo‑effective but not unitary flat  if and only if 
the inequality is strict for some (hence every) Gauduchon metric.\\

Suppose that $(1)$ holds. For any Hermitian metric form $\omega$ on $M$, let  $\omega_{g}=e^f\omega$ be a Gauduchon metric (\cite{Gau84}).   By $(1)$, for
every proper coherent subsheaf $\cF\subset E$, the following inequality holds:
\beq  \deg_{\omega_g}(E)-\deg_{\omega_g}(\cF)=\deg_{\omega_g}(\cQ)=
2\pi n\int_M c^{\mathrm{BC}}_1(\cQ)\wedge \omega_g^{n-1}>0,
\eeq
where $\cQ$ is the coherent quotient sheaf $\cQ=E/\cF$.
By Theorem \ref{main}, for any positive-definite Hermitian tensor $P_1\in \Gamma(M, E^*\ts \bar E^*)$,  there exists a Hermitian metric $h_0$ on $E$ such that 
$$\Lambda_{\omega_g}\left(\sq R^{h_0}\right)=P_1.$$
In particular, 
$$\Lambda_{\omega}\left(\sq R^{h_0}\right)=e^{f}P_1>0.$$
Hence, by \cite[Theorem~1.2]{XYY26+}, for any (quasi-)positive-definite Hermitian tensor $P \in  \Gamma(M, E^*\ts \bar E^*)$, there exists a unique Hermitian metric $h$ on $E$ such that 
\beq \Lambda_\omega\left(\sq R^{h}\right)=P. \eeq 

Conversely, we assume that $(2)$ is valid. Then  for any Gauduchon metric $\omega_g$ and for any (quasi-)positive-definite Hermitian tensor $P \in  \Gamma(M, E^*\ts \bar E^*)$, there exists a Hermitian metric $h$ on $E$ such that 
\beq \Lambda_\omega\left(\sq R^{h}\right)=P. \eeq 
Since the Hermitian-Yang-Mills tensors are non-decreasing along quotients, for any nonzero coherent quotient sheaf $\cQ$ of $E$,  one has
\beq 
\deg_{\omega_g}(\cQ)=2\pi n\int_M c^{\mathrm{BC}}_1(\cQ)\wedge \omega_g^{n-1}>0.\eeq 
In particular,  $\cQ$ is pseudo-effective but not unitary flat.
\eproof

\noindent The proof of Corollary~\ref{main2} proceeds along the same lines as the proof of Theorem~\ref{main4}.

\bproof[Proof of Theorem \ref{main5}] By \cite[Theorem~1.7]{Ou23},  if  $M$ is  a rationally connected  manifold and $-K_M$ is nef, then $TM$ is generically ample. In particular, by \cite[Theorem~1.4]{Ou23}  for any nonzero coherent quotient sheaf $\cQ$ of $TM$, the determinant line bundle $\det\cQ$ is pseudo-effective. The generic ampleness of $TM$ implies that there exists a generic curve $C\subset M$ such that $\det(\cQ)\cdot C>0$.  In particular,  $\det(\cQ)$ is  not unitary flat. Now the conclusion follows from Theorem \ref{main4}. On the other hand, if $M$ is a projective manifold with $-K_M$ strictly nef, it is established in \cite[Theorem~1.2]{LOY19} that $M$ is rationally connected and $-K_M$ is nef. This completes the proof.
\eproof 

\bremark It is natural to consider the boarder line case whether Theorem \ref{main5} holds when $M$ is rationally connected and $-K_M$ is pseudo-effective but non-trivial.
\eremark 

\bproof[Proof of Corollary~\ref{main7}] $(2)\Longrightarrow(1)$ is proved in \cite[Corollary~1.5]{Yang18}.  $(1)\Longrightarrow(2)$ is obtained in \cite[Theorem~1.7]{LZZ25}.  $(3)\Longrightarrow(2)$ is obvious. For $(2)\Longrightarrow(3)$, we suppose that  there exists a Hermitian metric $\omega$ on $M$ and a Hermitian metric  $h_0$ on $T^{1,0}M$ such that 
$ \Lambda_\omega\left(\sq R^{h_0}\right)>0$. 
By Theorem \ref{main} or \cite[Theorem~1.2]{XYY26+}, for any (quasi-)positive Hermitian tensor  $P$, there exists a unique Hermitian metric $h$ on $M$ such that 
\beq \Lambda_\omega\left(\sq R^{h}\right)=P. \eeq 
This completes the proof.
\eproof 


\end{document}